\documentclass[a4paper]{amsart}
\usepackage{amsfonts}
\usepackage[all]{xy}
\input {diagxy} 

\def\goth #1{{\mathfrak{ #1}}}

\def\set#1{\{\,#1\,\}}            
\def\>#1{{\bf #1}}                

\def\matriz#1#2{\left( \begin{array}{#1} #2 \end{array}\right) }
\def\Eq#1{{\begin{equation} #1 \end{equation}}}

\def\mitad{\frac{1}{2}}
\def\Tr{\hbox{{\rm Tr}}}

\newcommand{\R}{\mathbb{R}}    
\newcommand{\C}{\mathbb{C}}    

\renewcommand{\L}{\mathbb{L}}

\def\dim{\hbox{{\rm dim}}}              

\def\Ad{{\hbox{Ad}}}

\def\diag{{\hbox{diag}}}

\newtheorem{theorem}{Theorem}
\newtheorem{corollary}{Corollary}
\newtheorem{proposition}{Proposition}

\begin{document}
\title{Optimal Control Realizations of Lagrangian Systems with Symmetry}

\author{M. Delgado-T\'ellez}
\address{Depto. de Matem\'atica Aplicada Arquitectura T\'ecnica,  Univ. Polit\'ecnica de Madrid.
Avda. Juan de Herrera 6, 28040 Madrid, Spain}
\email{marina.delgado@upm.es}

\author{A. Ibort}
\address{Depto. de Matem\'aticas, Univ. Carlos III de
Madrid, Avda. de la Universidad 30, 28911 Legan\'es, Madrid, Spain.}
\email{albertoi@math.uc3m.es}

\author{T. Rodr\'{\i}guez de la Pe\~na}
\address{Depto. de Ciencias, Univ. Europea de Madrid.  C/ Tajo s/n, Villaviciosa de Od\'on, 28670, Madrid,
Spain.} \email{thaliamaria.rodriguez@uem.es}

\date{\today}
\thanks{This work has been partially
supported by the Spanish MCyT grants MTM2004-07090-C03-03, BFM2001-2272.}

\begin{abstract}
A new relation among a class of optimal control systems and
Lagrangian systems with symmetry is discussed. It will be shown
that a family of solutions of optimal control systems whose
control equation are obtained by means of a group action are in
correspondence with the solutions of a mechanical
Lagrangian system with symmetry. This result also explains the 
equivalence of the class of Lagrangian systems with symmetry and
optimal control problems discussed in \cite{Bl98}, \cite{Bl00}.

The explicit realization of this correspondence is obtained by a
judicious use of Clebsch variables and Lin constraints, a
technique originally developed to provide simple realizations of
Lagrangian systems with symmetry. It is noteworthy to point out
that this correspondence exchanges the role of state and control
variables for control systems with the configuration and Clebsch
variables for the corresponding Lagrangian system.

These results are illustrated with various simple applications.

\end{abstract}

\maketitle

\tableofcontents

\newpage

\section{Introduction}

A new insight on the properties of Lagrangian systems with symmetry has been gained by looking at
them from the point of view of optimal control theory (see for instance \cite{Bl00}) and,
conversely, a new representation for a class of optimal control problems and Pontryagin's maximum
principle was obtained in this way. In particular, it was shown that the rigid body problem and
Euler's equation for incompressible fluids, when formulated as optimal control problems, gave rise
to the symmetric body realization of the rigid body and the impulse momentum representation of
Euler's equations for inviscid incompressible fluids respectively. Thus, it was found that the
application of Pontryagin maximum principle to those optimal control problems leads to a nicer, more
symmetrical form of the system dynamical equations.

In this paper we will discuss the underlying geometrical structure common to these two examples,
showing that they are particular instances of a general correspondence among the
solutions of a particular class of optimal control problems, that will be called
Lie-Scheffers-Brockett optimal control systems, and Lagrangian systems with symmetry.

We will  say that a control system is of Lie-Scheffers-Brockett class\footnote{This class of
systems were considered first, obviously not as a control problem, by Lie and Scheffers \cite{Li93}
and much later, they where also discussed at length by R. Brockett \cite{Br70,Br73} already in the
realm of control theory.} if in local coordinates $x^i$ in state space, has the form:
$$\dot{x}^i = \sum_{a=1}^r c_a(t) X_a^i (x) ,$$
where in addition the vector fields $X_a = X_a^i \partial /\partial x^i$, satisfy
the Lie closure relations:
$$[X_a, X_b] = C_{ab}^c X_c ,$$
for some set of constants $C_{ab}^c$. Lie-Scheffers-Brockett systems are the prototypical examples
of dynamical systems defined on Lie groups or homogeneous spaces. Lie and Scheffers found \cite{Li93} that for
dynamical systems of this form there always exist non-linear superposition principles for the
composition of solutions of the differential equation, the most celebrated and well-known example
being the Riccati equation. Furthermore, if $M$ is a homogeneous space for the Lie group $G$ and we
denote by $\xi_a$ a basis of its Lie algebra $ \goth g$, we can consider the family of vector
fields $X_a$ induced on $M$ by them. A Brockett system on $M$ is just a vector field
$$\Gamma = \sum_{a=1}^r u^a(t) X_a $$
where $u^a$ denote the control variables. In this sense, Brockett theory of control systems on
homogeneous spaces are a global formulation of the systems considered by Lie and Scheffers (see
\cite{Ca00} and references therein for a detailed account of Lie-Scheffers theorem and its
applications).

On the other hand, Lagrangian systems with symmetry exhibit important structures that have been
widely used to study the qualitative structure of their solutions. It was the work on reduction of
symplectic systems with symmetry by J. E. Marsden and A. Weinstein \cite{Ma74} summarizing and
improving classical ideas on symmetry that opened the door to a systematic understanding of the
structures of some of the paradigmatic systems in mechanics and differential equations, for
instance the rigid body and Euler's equations. Soon it was realized that such systems can be
obtained by reduction of simple Lagrangian systems with symmetry usually defined on the tangent
bundle of a Lie group (finite or infinite dimensional) or, more generally, on a principal bundle
over some configuration space. This idea has been extensively used to describe a large class of
systems running from plasma physics to elasticity and other problems in continuum media. An
important contribution to the field has been the intrinsic understanding of the reduction of
Euler-Lagrange's equations and some of their most important related structures (see \cite{Ce01} and
references therein for a panoramic vision of these ideas).

Apart from reduction, an important tool in the study of Lagrangian systems with symmetry is
provided by Clebsch variables and Lin constraints. They were introduced as a way of providing a
variational derivation of Euler's equations in Eulerian variables \cite{La32,Li63,Ma83}. Later on,
in \cite{Ce87a,Ce87b} a geometrical framework was developed for this idea and nowadays new
applications in computational aspects of dynamics are arising (see for instance \cite{Co09,De09} and
references therein). If $L$ is a Lagrangian system with symmetry group $G$ defined for instance on
a principal fibre bundle $Q$ with structure group $G$, then a way to derive the equations of motion
on the quotient space $TQ/G$ is to use an auxiliary space $P$ where the group $G$ acts nonlinearly
in general. Then we choose on such space an appropriate subspace of curves. Such subspace is
selected by using a given connection $B$ on $Q$. It was proven in \cite{Ce87b} that the space of
horizontal curves (provided that a suitable technical condition is satisfied) is isomorphic to the
space of curves of the original variational principle defined by $L$. Moreover, the conditions
defining such subspace of horizontal curves can be expressed neatly as the vanishing of certain
differential condition, requirement that was first considered by Lin in the context of fluid
dynamics (for trivial bundles and connections) and, thereby, were named Lin constraints
\cite{Ce87b}. When using a suitable Lagrange's multiplier theorem to incorporate the constraints
into the Lagrangian density, it is found the expression described in Section  \ref{Clebsch}. It
will be called a Clebsch realization of the system $L$ with horizontal Lin constraints. The
auxiliary spaces used for this construction were originally $E\oplus E^*$, where $E$ is a linear
representation space for the group $G$. Variables on $E$ and $E^*$ were called Clebsch variables
and this terminology comes from the work by Clebsch providing a suitable representation for the
(eulerian) velocity field \cite{La32}. A key idea in this paper is that if we consider as
auxiliary space $T^*P$, the cotangent bundle of $P$, the variables $(x^i, p_i)$ on it can be
identified with state and costate variables of an optimal control problem. On the other hand, as it
was indicated above, the system thus obtained by using Clebsch variables and horizontal Lin
constraints, is equivalent to the original Lagrangian system with symmetry. In this way, we will
prove the equivalence between a Lagrangian system with symmetry and the appropriate optimal control
problem.

The paper is organized as follows: in Section \ref{LSB} Lie-Scheffers-Brockett optimal control
systems are presented. Afterwards, in Section  \ref{Clebsch}  we will discuss how optimal control
problems of Lie-Scheffers-Brockett type can be identified with a Clebsch realization of a
Lagrangian system.  Section  \ref{horizontal_curves} will be devoted to discuss the spaces of
curves where the variational principle determined by the Clebsch Lagrangian is defined.
In Section \ref{lin_cons}  it will be shown that the variational principle of an optimal control system of
Lie-Scheffers-Brockett type is equivalent to a Clebsch realization of a Lagrangian system defined
by the objective functional and with horizontal Lin constraints given by the control equation
itself. Then, the equivalence with a Lagrangian system with suitable end-point conditions is
established in Section \ref{main}. Finally, some simple applications and examples are discussed in Section \ref{examples}.

\newpage

\section{Lie-Scheffers-Brockett optimal control systems}\label{LSB}

Let $G$ be a Lie group acting on the right on a smooth manifold
$P$. We shall denote by $\goth g$ the Lie algebra of $G$ and by
$\xi \in \goth g$ a generic element. Then $\xi_P$ will denote the
Killing vector field  on $P$ associated to $\xi$ by the action of $G$,
i.e.,
\begin{equation}\label{xiP}
\xi_P(x) = \left.\frac{d}{dt} \Big(x\cdot \exp (-t\xi)\Big) \right|_{t = 0}, \quad \quad x \in P .
\end{equation}
If $\xi_a$ is a given basis in $\goth g$, $[\xi_a, \xi_b]= C_{ab}^c \xi_c$, and $X_a$ denotes
the corresponding vector field on $P$, i.e., $X_a = (\xi_a)_P$ with
the notation above.  We will also have that:
\begin{equation}\label{xi_P}
\xi_P = \sum_{a=1}^r u^a X_a ,
\end{equation}
where $\xi = u^a \xi_a$ and,
\begin{equation}\label{lie_algebra}
 [X_a,X_b] = C_{ab}^c X_c .
\end{equation}
If we let the coordinates $u^a$ on the Lie algebra $\goth g$ be
time-dependent functions, they can be interpreted as control
variables controlling the system $\Gamma :=\xi_P$.

We shall consider the non-autonomous dynamical system on $P$ defined by the vector field $\xi_P$,
i.e., \Eq{\label{xiS} \dot{x}^i = \xi_P (x)  = \sum_{a=1}^r u^a(t) X_a^i(x) ,}and we will interpret
it as the state equation for a control system with state space $P$.

We will further restrict the class of systems we are interested in by introducing another structure
that will provide a particular realization of the control variables $u^a$. Suppose that the Lie
group $G$ acts on the left on a smooth manifold $Q$. We shall assume for simplicity that the action
is proper and free, hence the orbit quotient space $Q/G$ is a smooth manifold $N$ and the canonical
projection map $\pi_2 \colon Q \to N$ is a submersion. Therefore, the map $\pi_2$ defines a principal
fibration over $N$ with structure group $G$.

Let $B$ be a principal connection on the principal bundle $Q(G,N)$, this is, $B$ is a $\goth
g$-valued 1-form $B\colon TQ\to\goth g$ on $Q$ such that $B(\xi_Q(q))=\xi$, $\forall \xi\in \goth
g$, and $L_g^* B = \Ad_g B$, $g\in G$, where $\Ad_g$ denotes the adjoint action of $G$ on $\goth g$
and $L_g$ the left action of $G$ on $Q$. The restriction of a connection to the tangent space
$T_qQ$ is denoted as $B_q$. Then we can define a map $\xi\colon TQ \to \goth g$ by means of:
\Eq{\label{controlrep} \xi(q,\dot{q}) = B_q(\dot{q}) ,   \quad \quad (q,\dot{q}) \in T_qQ .} 
If we have now a Lie-Scheffers-Brockett system of the form, 
\Eq{\label{controlrepres} \dot{x} =\xi(q,\dot{q})_P (x) ,}
we can consider it as a control system whose control variables are $u = (q,
\dot{q}) \in TQ$.  In local coordinates $(q^i,\dot{q}^i)$ on $TQ$,
the map $\xi$  will have the form:
$$\xi(q,\dot{q}) = \dot{q}^i B_i(q) =
\dot{q}^i B_i^a(q) \xi_a ,$$ 
and if we describe the vector fields
$X_a$ in local coordinates $x^\alpha$ on $P$ as
$$ X_a = \xi_a^\alpha (x) \frac{\partial}{\partial x^\alpha} ,$$
we finally get for the control system above the following expression in local coordinates:
$$ \dot{x}^\alpha = \dot{q}^i B_i^a(q) \xi_a^\alpha (x) .$$

Finally, if $L\colon P\times TQ\to \R$ is a function on $P\times TQ$, we can construct the
objective functional 
\Eq{\label{objective} S_L[x,q] = \int_0^T L(x,q,\dot{q}) dt ,}defined on a suitable
space of curves $\Omega$ on $P \times Q$.  We restrict our attention to $G$-invariant
objective functionals, i.e., those $S$ defined by densities $L\colon P \times TQ\to \R$ which are
$G$-invariant functions on $P\times TQ$. Notice that $G$ acts on the left on $TQ$ by lifting the
action on $Q$ to $TQ$. The quotient space $P\times TQ/G$ can be identified, using an auxiliary connection,
with the pull-back to $TN$ of the bundle $P\times {\rm ad}Q$, where ${\rm ad} Q$ is the $\goth
g$-algebra bundle adjoint of $Q\to N$ that can be defined as $Q\times_G \goth g = Q\times \goth g
/G$ with natural local coordinates $(n^\alpha,\dot{n}^\alpha,\xi^a )$, where $n^\alpha$ are local
coordinates on $N$ and $\xi^a$ on $\goth g$.

Finally, fixing the endpoint conditions \Eq{\label{endpoint} x_0 = x(0), \quad x_T = x(T) ,}we will
consider the optimal control problem on the state space $P$ with control equation
(\ref{controlrepres}), objective functional (\ref{objective}), and the fixed endpoint conditions
above (\ref{endpoint}).

Provided that the manifold $Q$ is boundaryless, it is well-known
that if the curve $(x(t), q(t), \dot{q}(t))$ is a normal extremal
of the optimal control problem  (\ref{controlrepres}), (\ref{objective})  and (\ref{endpoint}), then there exists an
extension $(x(t), p(t))$ of the state trajectory $x(t)$ to the
costate space $T^*P$, satisfying Pontryagin equations:
\begin{equation}\label{MPP} \dot{x} =\xi(q,\dot{q})_P (x), \quad   \dot{p} = -\frac{\partial H_P}{\partial x}, \quad  \frac{\partial H_P}{\partial q} = \frac{\partial H_P}{\partial \dot{q}} = 0 ,
\end{equation}
where $H_P$ is Pontryagin's hamiltonian function,
$$ H_P(x,p,q,\dot{q}) = \langle p, \xi(q,\dot{q})_P (x) \rangle - L(x, q,\dot{q}) .$$
In more geometrical terms, normal extremals are
integral curves of the presymplectic system $(M_P,\Omega_P, H_P)$
where $M_P = T^*P\times TQ$ and $\Omega_P$ is the presymplectic form
obtained by pull-back to $M_P$ of the canonical symplectic form
$\omega_P$ on $T^*P$ (see for instance \cite{De03}, \cite{Ib10}).

\section{Clebsch representation of Lie-Scheffers-Brockett optimal control systems}\label{Clebsch}

Normal extremals of the optimal control problem above are critical paths of the objective
functional (\ref{objective}) in an appropriate space of curves $\Omega (P \times Q ; x_0, x_T)$
subjected to the  restriction imposed by eq. (\ref{controlrepres}).     
We will assume on this article that all curves and functions are smooth, 
which is consistent with the geometrical framework we are using and 
we introduce the constraints defined by the control equation (\ref{controlrep}) as Lagrange multipliers. 
In Appendix A we discuss and prove a version of Lagrange multipliers theorem which is suitable for this setting (see Thm. \ref{LMT} and Thm. \ref{smooth_LMT} in for details).   Thus,  if $(x(t), q(t))$ is a smooth
extremal for the objective functional (\ref{objective}) satisfying the control equation
(\ref{controlrep}), the lifted curve $(x(t),p(t))$ will be smooth and  $(x(t), p(t), q(t))$ will
be a critical path of the extended functional
 \Eq{\label{SL} S_{\L} [x,p,q]= \int_0^T \left( L(x,q,\dot{q})
+ \langle p, \dot{x} - \xi(q,\dot{q})_P (x)
    \rangle \right)  dt }in the space of smooth curves $\gamma(t) = (x(t),p(t),q(t))
\in \Omega_{x_0, x_T} (T^*P\times Q)$  with fixed endpoints $x_0$,
$x_T$.

A simple computation allows us to write the Lagrangian density
$\L\colon T(T^*P) \times  TQ\to \R$ of the functional (\ref{SL}),
as
\begin{eqnarray}\label{LClebsch} \L (x,p,\dot{x}, \dot{p}, q,\dot{q}) =
    L(x,q,\dot{q}) - \langle p, (B_q(\dot{q}))_P(x)   \rangle + \langle p, \dot{x} \rangle
    =\\\nonumber= L(x,q,\dot{q}) + \langle J(x,p), B_q(\dot{q}) \rangle
    +
    \langle \theta_P(x,p), \dot{x}\rangle .
 \end{eqnarray}%
The map $J\colon T^*P \to {\goth g}^*$ in (\ref{LClebsch}) denotes the momentum map associated to
the cotangent lifting of the action of $G$ to $P$, this is, $\langle J(x,p), \xi \rangle = \langle
p, \xi_P(x) \rangle$, for all $\xi \in \goth g$, $(x,p)\in T^*P$. The 1-form $\theta_P$ denotes the
canonical Liouville 1-form on $T^*P$ with the following expression in local coordinates $\theta_P =
p_\alpha dx^\alpha$.

We will show in the following sections that the expression on the r.h.s. of eq. (\ref{LClebsch})
has the form of a Clebsch Lagrangian. We summarize the discussion so far in the following
statement:

\begin{theorem}\label{thOpti}
Let $x_0,x_T$ be two fixed points in $P$ and $q_0$ on $Q$, then the following assertions are
equivalent:
\begin{itemize}
\item[i.-]  The smooth curve $(x(t),q(t))$ is a
critical point of  the objective functional:
$$  S_L [x,q] = \int_0^T L(x,q,\dot{q}) dt ,$$
on the space of smooth curves satisfying the equation
$\dot{x}=\xi(q,\dot{q})_P (x)$, and $x(0) = x_0$, $x(T) = x_T$,
$q(0) = q_0$.

\item[ii.-]  There exists a smooth lifting $p(t)$ of the curve
$x(t)$ to $T^*P$ such that the curve $(x(t),p(t),q(t))$ is
a critical  point of the functional:
$$ S_{\L}[x,p,q] = \int_0^T \left( L(x,q,\dot{q}) + \langle J(x,p), B_q(\dot{q}) \rangle + \langle \theta_P(x,p), \dot{x}\rangle \right) dt .$$

\end{itemize}

\end{theorem}

\begin{proof} We will consider the variational principle (i) as the problem of determining the critical points of
the functional $S_L$ in the subspace of smooth curves on $P \times Q$ with fixed
endpoints $x(0) = x_0$, $x(T) = x_T$ and $q(0) = q_0$ and satisfying the constraint
$\dot{x}=\xi(q,\dot{q})_P (x)$. Now, because of the Lagrange multipliers theorem, (Appendix A, Thm. \ref{LMT}) and eq. (\ref{LClebsch}) that allows to write the
Lagrange multiplier $\langle p, \dot{x} - (B_q(\dot{q}))_P(x)   \rangle$
as the Lagrangian density in (ii), the critical points of (i) are the same as those of (ii).
\end{proof}

We will call the Lagrangian $\L$ the Clebsch Lagrangian for the Lie-Scheffers-Brockett optimal
control problem stated in Section \ref{LSB}.

\section{Spaces of horizontal curves on associated bundles}\label{horizontal_curves}

One of the consequences of the previous discussion regarding
solutions of a Lie-Scheffers-Brockett optimal control problem and
critical points of a Clebsch Lagrangian is that the later is
equivalent to a Lagrangian system with symmetry.  It was
established in \cite{Ce87b} that the critical points of a Lagrangian system with
symmetry on $TQ$ are in one-to-one correspondence with the critical points of an
appropriate Clebsch Lagrangian representation of the system. We
will present here a similar result that is adapted to the
setting used in this paper. Before giving a precise statement
of the result we need to discuss some background material about spaces of
horizontal curves.

\subsection{The associated bundle $P\times_G Q$}\label{associated}

As in previous sections $P$ will denote a right $G$-space
whose space of orbits $P/G$ will be denoted by $M$.   We will assume that $M$ is
a smooth manifold and that the canonical projection map $\pi_1\colon P \to M$ is a submersion.
Finally it will be assumed that $Q$ is a left
$G$-principal bundle with base manifold $N$ and projection map  $\pi_2\colon Q \to N$.

Given the principal bundle $Q$ and the
action of $G$ on $P$, we can construct the associated bundle over
$N$ with fibre $P$ as follows: the group $G$ acts on $P\times Q$
on the left as $g\cdot (x,q) = (xg^{-1}, gq)$, for all $(x,q)\in
P\times Q$, $g\in G$. Because $G$ acts freely on $Q$, then $G$
acts freely on $P\times Q$. We shall denote the equivalence class
$[(x,q)]$ defined by the orbit of $G$ passing through the pair
$(x,q)$ simply as $xq$, i.e.,
$$xq = \set{(xg^{-1},gq) \in P\times Q \mid g \in G } .$$
Notice that with this notation we have the ``associative'' property $(xg)q = x(gq)$ for all $x\in P$, $g\in G$ and $q \in Q$.

The projection $\pi_2$ induces another projection $\pi_P\colon
P\times_G Q \to N$, $\pi_P(xq) = \pi_2(q)$. Then it is clear that
$P\times_G Q$ is a fiber bundle over $N$ with fiber $P$ and
projection $\pi_P$. Given a point $q\in Q$, there is a natural map
$i_q\colon P \to P\times_G Q$, defined by $i_q(x) = xq$ that maps
$P$ into the fibre of $\pi_P$ over the point $\pi_2(q)$.

Besides $\pi_P$ there is another natural projection $\pi_Q$ on
$P\times_G Q$ induced by the projection $\pi_1\colon P\to M$,
and defined by $\pi_Q(xq) =
\pi_1(x)$. Again, for any $x\in P$, there is a natural map
$i_x\colon Q\to P\times_G Q$ defined by $i_x(q) = xq$, that maps
$Q$ into the fibre of $\pi_Q$ over the point $\pi_1(x)$.  The diagramme below
summarizes the spaces and projections introduced above.

$$
{ \bfig\xymatrix{ & P \times Q\ar[dl]_{p_1}\ar[dr]^{p_2}\ar[d]_\Pi & \\ P\ar[d]_{\pi_1} & P\times_G Q\ar[dl]^{\pi_Q}\ar[dr]_{\pi_P} & Q\ar[d]^{\pi_2}  \\ M & & N
}\efig} 
$$

Tangent vectors $v\in T_{xq}(P\times_GQ)$ can be nicely described
as follows. Let $x(t)q(t)$ be a curve such that $x(0)q(0)= x_0q_0$
and $d/dt(x(t)q(t))\mid_{t= 0} = v$. Then it is not difficult to
show that $v =  T_qi_x(\dot{q}) + T_xi_q(\dot{x})$.   
We will use a simplified and convenient notation for tangent vectors
$v\in T_{xq} (P\times _G Q)$ following the conventions above, and
we simply denote $x\dot{q} := T_qi_x(\dot{q})$ and $\dot{x}q :=
T_xi_q(\dot{x})$. With this notation we have:
\begin{equation}\label{dot_xq}
 \dot{\overline{xq}} := \left.\frac{d}{dt} \Big(x(t)q(t) \Big) \right|_{t = 0} = x\dot{q}+ \dot{x}q.
 \end{equation}

\subsection{Induced connections on the associated bundle $P\times_G Q$}\label{associated_connections}

Recall that a principal connection $B$ on $Q$ is
characterized by its vertical and horizontal spaces at $q\in Q$.
They are denoted respectively by $V_q=\ker T_q\pi_2$, $H_q=\ker B_q$,
and they provide a
decomposition $T_qQ = H_q \oplus V_q$.   Notice that  the map
$T_q\pi_2\colon H_q\to T_{\pi_2(q)}N$ is an isomorphism.  
We denote by $H(TQ)=\bigcup_{q\in
Q}H_q$ and $V(TQ)=\bigcup_{q\in Q}V_q$ the corresponding invariant subbundles under
the action of $G$ on $Q$.  Then we have $TQ=H(TQ)\oplus V(TQ)$. 

The vertical and horizontal components of a vector $v\in T_qQ$
will be denoted by $V(v)$ or $v^V$, and $H(v)$ or $v^H$ respectively. By definition,
$V(v)=B_q(v)_Q$ and $H(v)=v-B_q(v)_Q$.  A tangent vector $v$ is
called horizontal if its vertical component is zero; i.e., if
$B_q(v)=0$.  The vector $v$ is called vertical if its horizontal component
is zero; i.e., $T_q\pi_2(v)=0$. A curve $q(t)$ will be said to be
horizontal if $\dot{q}(t)$ is horizontal for all $t$.  Hence, a
curve $q(t)$ on $Q$ is horizontal if $B_{q(t)}(\dot{q}(t))=0$, for all $t$.

Given a vector $X\in T_{\pi_2(q)}N$ the horizontal lift $X_q^H$ of
$X$ at $q$ is the unique horizontal vector in $T_qQ$ such that
$T_q\pi_2 (X_q^H) = X$.

For any smooth curve $n(t)\in N$, $t\in [0,T]$ and $q_0 \in Q$ with  $\pi_2
(q_0) = n_0 = n(0)$ we define its horizontal
lift $n^h_{q_0}(t)$ as the unique horizontal curve
projecting onto $n(t)$ and such that $n^h_{q_0}(0) = q_0$.

Consider a smooth curve $q(t)\in Q$,$\ t\in\,[0,T]$. Then there is a
unique horizontal curve $q^h(t)$ such that $q^h(t_0)=q(t_0)$ and
$\pi_2(q^h(t))=\pi_2(q(t))$ for all $t\in[0,T]$. Therefore, there is a unique smooth curve $g(t)$, $t\in [0,T]$ in $G$
such that $q(t)=g(t)q^h(t)$.  Also, notice that if we denote by $n(t)=\pi_2
(q(t))$ then $q^h(t) = n_{q_0}^h(t)$.  Then
$\dot{q}(t)=\dot{g}(t)q^h(t) + g(t)\dot{q}^h(t)$ where $\dot{g}(t)$ denotes the 
tangent vector on $G$ to the curve $g(t)$ at time $t$ and $\dot{g}(t)q^h(t)$
will denote the tangent vector $(\dot{g}(t))_Q(g(t)q^h(t))$.  
Similarly, $g(t)\dot{q}^h(t)$ will denote the horizontal tangent vector on $Q$ obtained by translating
the horizontal tangent vector $\dot{q}^h(t)$ by the action of the element of the group $g(t)$.
On the other hand we have the decomposition of the tangent vector
\begin{equation}\label{hor_ver}
\dot{q}(t)=\dot{q}^H(t) + \dot{q}^V(t); \quad \quad \dot{q}^H (t) = g(t) \dot{q}^h(t), \quad \dot{q}^V(t) = \dot{g}(t) q^h(t) .
\end{equation}
By definition of a horizontal vector,
$B_{q(t)}(g(t)\dot{q}^h(t))=0$, thus
$$B_{q(t)}(\dot{q}(t))=B_{q(t)}(\dot{g}(t)q^h(t))
=B_{q(t)}(\dot{g}(t)g^{-1}(t) q(t))=\dot{g}(t)g^{-1}(t) .$$

The principal connection $B$ on $Q$ induces a connection on any associated bundle
by defining the horizontal space to be the space spanned by tangent vectors to all curves of the form $x_0q(t)$ where
$q(t)$ is horizontal.   Such curves will be called horizontal.  This defines a distribution $H^P$ on
$T(P\times_GQ)$, this is $v\in H^P_{xq}$ if there exists a horizontal
curve $\gamma(t)=x(t)q(t)$, $\gamma(0) = xq$ on $P\times_GQ$ such that
$v=\dot{\gamma}(0)$.
Hence, because of the definition of the associated connection, horizontal
vectors in $H^P$ will have the form $x \dot{q}^h$.

Notice that if $\xi\in \goth g$ is an element on the Lie algebra
of $G$, then for each $t$ we have:
$$ (x\cdot \exp (t\xi ) )q = x(\exp(t\xi)\cdot q) . $$
Taking derivatives and using eq. (\ref{xiP}) we obtain,
\Eq{\label{balance} x \xi_Q(q) = - \xi_P(x) q ,} 
where $\xi_Q(q) = d/dt (\exp (t\xi ) \cdot q)\mid_{t = 0}$.

Given $x(t)$ and $q(t)$ curves in $P$ and $Q$ respectively and using eqs. (\ref{dot_xq}) and (\ref{hor_ver}) we
have:
$$\frac{d}{dt} (x(t)q(t))  =\dot{x}(t)q(t)+ x(t)g(t)\dot{q}^h(t) +
x(t)B_{q(t)}(\dot{q}(t))_Q (q(t)) .$$ 
Then the tangent vector
$\dot{\overline{x(t)q(t)}}$ will be horizontal if and only if $$x(t)
B_{q(t)}(\dot{q}(t))_Q (q(t)) + \dot{x}(t)q(t) = 0.$$
Thus we can
define a connection 1-form $B^P$ on the associated bundle
$\pi_Q \colon P\times_G Q \to N$ with values in its vertical subbundle, $V(\pi_Q) = \ker T\pi_Q$, whose kernel is the horizontal distribution
$H^P$ defined above, given by:
\begin{equation}\label{connectionB}
 B_{xq}^P (\dot{\overline{xq}}) = x B_q(\dot{q})_Q(q) + \dot{x}q = (\dot{x} - B_q
(\dot{q})_P(x))q ,
\end{equation}
where we have used eq. (\ref{balance}) in the last equality of the previous formula.

\subsection{Horizontal curves in $P\times Q$}

We will denote the space of smooth curves in $P$ with fixed origin
$x_0$ by $\Omega_{x_0}(P)$ and the space of smooth curves with fixed
endpoints $x_0,x_T$ by $\Omega_{x_0,x_T}(P)$. Likewise the space
of smooth curves in $Q$ with origin $q_0$ will be denoted by
$\Omega_{q_0}(Q)$ and the space of smooth curves with fixed endpoints
$q_0,q_T$ will be denoted by $\Omega_{q_0,q_T}(Q)$.   All these spaces of curves
define regular submanifolds of the Hilbert manifold of curves of Sobolev class $k\geq 1$
on $P$ or $Q$ respectively as it is discussed in Appendix B.

As it was stated above it is clear that given a curve $q(t)$ in $Q$ there is
a unique decomposition $q(t)=g(t)q^h(t)$, where $g(0)=e$ and
$q^h(t)$ is horizontal with respect to the connection $B$, i.e.,
$B_{q(t)}(\dot{q}_h(t)) = 0$.  

Given a curve $q(t)\in \Omega_{q_0}(Q)$ and $x_0\in P$, there exists a unique curve
denoted $x^h(t)$ in $\Omega_{x_0}(P)$ such that $x^h(t) q(t)$ is
horizontal with respect to the induced  connection $B^P$ on $P\times_G Q$
and satisfying $x^h(0)q(0) = x_0 q_0 $. It is easily seen that this curve is
defined by $x^h(t)=x_0 g^{-1}(t)$, because
$$ x^h(t)q(t) = x_0 g^{-1}(t)q(t) = x_0 {g}^{-1}(t)g(t)q^h(t)= x_0 q^h(t) ,$$
which is horizontal.  The space of horizontal curves $x(t)q(t)$ in $P\times_GQ$ 
with respect to the affine connection $B^P$ with initial value $x_0q_0$ will be 
denoted as $\Omega^H_{x_0q_0}(P\times_GQ)$ as it is a regular submanifold of
the space of curves $\Omega_{x_0q_0}(P\times_GQ)$ as it is discussed in Appendix B.

Similarly we will denote by
$\Omega^H_{x_0;q_0}(P\times Q)$ the set of  smooth curves
$(x(t),q(t))$ with domain $[0,T]$, $q(t)\in \Omega_{q_0}(Q)$,
$x(t)$ with initial point $x_0$ and such that $x(t)q(t)$
is horizontal for all $t$.  We will call such space the space of horizontal curves in $P\times Q$
with initial points $x_0, q_0$.  Notice that if $(x(t), q(t)) \in \Omega^H_{x_0,q_0}(P\times Q)$, then
the curve $(x(t)g(t)^{-1}, g(t)q(t))$ is horizontal too for any curve $g(t)$ on $G$ such that $g(0) = e$.
Thus the natural projection $\Pi\colon \Omega^H_{x_0;q_0}(P\times Q) \to \Omega_{x_0q_0}^H(P\times_GQ)$
defined as $\Pi(x(t),q(t)) = x(t)q(t)$ is a principal fibration with structure group the group of smooth curves $\Omega_e(G)$ on
$G$ starting at the neutral element $e$. 

Thus the assignment $q(t) \mapsto (x(t) = x_0g^{-1}(t), q(t))$ determines a one--to--one correspondence among the space $\Omega_{q_0}(Q)$ of smooth curves starting at $q_0$ and the space $\Omega^H_{x_0;q_0}(P\times Q)$ of horizontal curves  above.   

Finally given a curve $(x(t), q(t))$ in $\Omega^H_{x_0;q_0}(P\times Q)$, we have the curves $n(t) = \pi_2(q(t))$ in $\Omega_{n_0}(N)$ and $x(t)$ in $\Omega_{x_0}(P)$.    If the group $G$ acts transitively on $P$
it is easy to see that such correspondence is surjective because given $n(t)$ and $x(t)$ as above, we can define the curve $q(t) = g(t) n_{q_0}^h(t)$ where $x(t) = x_0g^{-1}(t)$.   The curve $q(t)$ thus constructed is in  $\Omega^H_{x_0;q_0}(P\times Q)$.
Notice that $g(t)$ exists because of the transitivity of the action of $G$ on $P$, however it is not unique.
 
This last requirement
implies that $x(t)g(t)  = x_0$, where $q(t) = g(t) q_h(t)$, and
that $x_T = x(T) = x_0 g^{-1}(T) = x_0 g_T^{-1}$, where we are denoting $g(T) = g_T$.

\section{Lin constraints and spaces of horizontal curves}\label{lin_cons}

\subsection{Compatible end-point conditions}

Given an initial condition $x_0 \in P$ for the control system
(\ref{xiS}), we will denote by $x_0G$ the orbit of the group $G$ on $P$
passing through $x_0$, i.e., $x_0G = \set{x_0g \in P \mid g \in
G}$.   

Given a smooth curve of controls $u(t)$, we denote as in eq. (\ref{xi_P}),
by $\xi_P(t;u)$ the time-dependent vector field on $P$ defined by
$\sum_{a = 1}^r u^a(t) X_a$ and by $x(t;u)$ the corresponding integral curve starting at $x_0$. 
The end-point $x_T= x(T;u)$  will lie in the orbit $x_0G$ for any finite $T$.
In fact,  there exists a (in general non-unique) smooth curve $\xi(t)$ on the Lie algebra
$\goth g$ of $G$ such that $\xi(t)_P = \xi_P(t;u)$. 
The non-uniqueness of the choice of the lifted curve $\xi(t)_P$
depends on the isotropy algebra ${\goth g}_{x(t)}$ along the
integral curves of the control vector field.  Thus if we denote by $\mathfrak{g}_t$ the subalgebra of  the Lie algebra $\mathfrak{g}$ generated by the elements $\xi$ such that $\xi_P =   \xi_P(t;u)$ for the given time $t$, then the collection of all $\mathfrak{g}_t$, $0\leq t \leq T$, defines a trivial bundle over $[0,T]$ as well as the collection of all ${\goth g}_{x(t)}$, $0\leq t \leq T$, and all that it takes is to choose a smooth section of their quotient bundle.

Then we integrate the differential equation
$$ \dot{g}(t) = \xi(t)g(t)$$
on the group $G$ with initial condition $g(0) = e$ on the interval $[0,T]$.   Denoting by
$g(t;u)$ such integral curve, we then get that $x_T = x_0 g(T;u)$ proving that
$x_T \in x_0G$.   The integral curve $g(t;u)$ is given explicitly in terms of the chronological exponential map as $g(t;u) = \mathrm{Exp}(\int_0^T \xi(t) dt )$.  (Notice that such integral exists on compact sets because the curve $\xi(t)$ is smooth, hence continuous.)

 The simple computation below shows that the curve $x(t) = x_0
g(t;u)$ is an integral curve of the control system (\ref{xiS})
with initial condition $x_0$,
$$ \frac{d}{dt}{x(t)} = \frac{d}{dt}({x_0 g(t;u)}) =x_0\dot{g}(t;u)=x_0(\xi(t)g(t;u))=\xi_P(t;u)(x(t)) ,$$
with $x_0 g(0;u) = x_0e = x_0$.

Thus for the endpoint $x_T$ to be accesible from $x_0$ it is
necessary that $x_T\in x_0G$. If $x_T \in x_0G$, we will also say
that the endpoints $x_0$ and $x_T$ are compatible.

However not any point $x_T \in x_0G$ is accessible from $x_0$ for general $G$ even though for connected groups, if $x_0, x_T$ are compatible, then $x_T$ is accesible from $x_0$. In fact the following proposition 
can be proved easily:

\begin{proposition}
If $G$ is connected, any point $x_T\in x_0G$ is
accesible from $x_0$.
\end{proposition}

\begin{proof} If the Lie group $G$ is connected then is arc--connected.  Take now
any smooth curve $g(t)$ such that $g^{-1}(0) = e$ and $g^{-1}(T) =
g^{-1}_T$ where $x_T = x_0 g^{-1}_T $. Now the curve $x(t) = x_0
g^{-1}(t)$ satisfies that $x(0) = x_0$ and $x(T) = x_T$.
Moreover,\\ $\dot{x}(t)q(t) = - x(t)\dot{g}(t) g^{-1}(t)q(t) =-
x(t)B_{q(t)}(\dot{q}(t))_Q(q(t))=B_{q(t)}(\dot{q}(t))_P(x(t))q(t)$
where $q(t) = g(t) q_0 $.   Consider then the vector field $\xi_P(t;u)) = B_{q(t)}(\dot{q}(t))_P(x)$ 
where $u(t) = (q(t), \dot{q}(t))$.  Now $\xi(t) = B_{q(t)}(\dot{q}(t))$ and the proposition is proved.
\end{proof}

\subsection{Spaces of horizontal curves with fixed end--point conditions}

Let us fix $x_0, x_T$ two compatible endpoints in $P$.  Let $g_T \in G$ be such that $x_T = x_0g_T^{-1}$. 
If $G_{x_0}$ denotes the isotropy group of $x_0$, i.e., $ G_{x_0}= \set{h \in G \mid x_0h = x_0 }$,  notice that for any $h \in G_{x_0}$, the group element $g_T ' = g_T h$
defines the same end-point $x_T = x_0 g_T'^{-1}$.

Consider a point $q_0\in Q$.  Thus if $q(t)$ is a
curve on $Q$ such that $x(t)q(t)$ is horizontal with respect to the affine connection
$B^P$, eq. (\ref{connectionB}), induced by the principal connection $B$ on $Q$, we have $x(t)q(t) = x_0
q^h(t)$. Thus $x_T q(T)=x_0 q^h(T)$, hence $x_0 g^{-1}_T q(T)=x_0
q^h(T)$ and $g^{-1}_T q(T) =g_0 q^h(T)$, $g_0 \in G_{x_0}$.
Moreover if $g(t)$ is the curve on $G$ such that $q(t) = g(t) q^h(t)$, then  $q(T) = g(T)q^h(T)$, and as $G$ acts freely on $Q$, we will conclude that $g(T) \in g_T G_{x_0}$.

Thus given $x_0$, $x_T = x_0g_T^{-1}$, and denoting $q^h(T)$ by $q^h_T$ for the curve $q(t) \in \Omega_{q_0, g_T G_{x_0} q_T^h} (Q)$, i.e., such that $q(0) = q_0$ and $q(T) \in g_T G_{x_0} q_T^h$, we can associate to it a unique curve in $\Omega^H_{x_0,x_T;q_0}(P\times Q)$ by means of the
natural correspondence $q(t) \mapsto (x(t),q(t))$, where $x(t) =
x_0 g^{-1}(t)$ and $q(t) = g(t)q^h(t)$ as before. Then $x(0) = x_0
g^{-1}(0) = x_0$ and $x(T) = x_0 g^{-1}(T) =x_0 h^{-1}g_T^{-1}=
x_0 g_T^{-1}= x_T$, $h\in G_{x_0}$.

Finally, notice that the space of horizontal curves $\Omega^H_{x_0,x_T;q_0}(P\times Q)$ is a closed submanifold of the space of horizontal curves $\Omega^H_{x_0;q_0}(P\times Q)$ obtained as the level set $\mathrm{ev}_T^{-1}(x_T)$, of the evaluation map $\mathrm{ev}_T\colon \Omega^H_{x_0;q_0}(P\times Q) \to P$, $\mathrm{ev}_T (x,q) = x(T)$, 

The previous remarks and commentaries can be summarized in the
following:

\begin{proposition}\label{equivalence1}   Let $Q$ be a left--principal $G$-bundle, 
$\pi_2\colon Q \to N$, with connection $B$ and
$P$ a right $G$-space.   Let $g_T\in G$, and two compatible points $x_0, x_T = x_0 g_T^{-1}$ in $P$, 
$q_0 \in Q$, $n_0 = \pi_2(q_0)$ and $n_T\in N$.   Then there is a
one-to-one correspondence among the space of parametrized smooth horizontal curves
$(x(t), q(t)) \in \Omega^H_{x_0,x_T;q_0}(P\times Q)$ and the set of curves
$q(t) \in \Omega_{q_0,g_T G_{x_0} q_T^h}(Q)$, where $q_T^h  = n_{q_0}^h(T)$,
$n_{q_0}^h(t)$ being the horizontal lifting starting at $q_0$ of $n(t)$, and $G_{x_0}$ the isotropy group of $x_0$. 
There is also a surjective correspondence among the set of horizontal curves $(x(t), q(t)) \in \Omega^H_{x_0,x_T;q_0}(P\times Q)$ and the set of curves $(x(t),n(t)) \in \Omega_{x_0,x_T}(x_0G)\times \Omega_{n_0,n_T}(N) $ where $x_0G \subset P$ is the $G$--orbit of $x_0$.   If $G_{x_0} = e$, then the later correspondence is one--to--one too.

\end{proposition}

\begin{proof}  The correspondence $\alpha \colon \Omega_{q_0,g_T G_{x_0} q_T^h}(Q) \to  \Omega^H_{x_0,x_T;q_0}(P\times Q)$ is given by:
\begin{equation}
\alpha(q(t))= ( x(t), q(t)) ,
\end{equation}
with $x(t) = x_0 g^{-1}(t)$ and $q(t) =  g(t) q^h(t)$,
and the correspondence $\beta\colon  \Omega^H_{x_0,x_T;q_0}(P\times Q) \to  \Omega_{x_0,x_T}(x_0G)\times \Omega_{n_0,n_T}(N)$, is explicitly given as: 
\begin{equation}
\beta((x(t),q(t)))  =   (x(t),\pi_2(q(t))) .
\end{equation}
It is clear that the map $\alpha$ is bijective, the inverse given simply by $(x(t),q(t)) \mapsto q(t)$.
From the definition is clear that $\beta$ is surjective.  A right inverse of the map $\beta$ is given by
$(x(t),n(t)) \mapsto q(t) = g(t)n^h_{q_0}(t)$, where $x(t) = x_0 g(t)$.   Now, if $G_{x_0} = e$, then the 
curve $g(t)$ in $G$ is uniquely determined and the right inverse of $\beta$ is unique.
\end{proof}

\subsection{Lin constraints}

Since $B^P\left(\frac{d}{dt}(x(t)q(t))\right)=0$ is equivalent to $x(t)q(t)$
being horizontal, it follows that $\Omega^H_{x_0,x_T;q_0}(P\times
Q)$ is the subset of $\Omega_{x_0,x_T;q_0}(P\times Q)$ defined by the constraint
$B^P(\frac{d}{dt}(x(t)q(t)))=0$. This constraint is called a
Lin constraint \cite{Ce87a}. Now we will introduce the constraint
in the variational principle using a Lagrange multiplier (see Apendixes A, B). Using
the costate space $T^*P$ we will allow arbitrary
variations of the curves along the vertical directions on $T^*P$
and the Lagrange multiplier will have the form:
\begin{equation}\label{lin_constraint}
\left\langle B^P\left(\frac{d}{dt}\left(x(t)q(t)\right)\right),p(t)q(t) \right\rangle
\end{equation}
with $p(t)$ being any lifting of the curve $x(t) \in
\Omega_{x_0,x_T}(P)$. Notice that the action of $G$ on $P$ can be
lifted naturally to an action of $G$ on $T^*P$, then we consider the
quotient space $T^*P\times_G Q$ as in Section \ref{associated}. We
take the curve $p(t)$ as a curve in $T^*P$ with endpoints
lying on $\tau_P^{-1}(x_0)$ and $\tau_P^{-1}(x_T)$, this is $p(t)$ is a curve
in $T^*P$ with free endpoints along the fibers of the canonical
projection $\tau_P \colon T^*P \rightarrow P$ projecting over the points $x_0$ and $x_T$. 
Because of eq. (\ref{connectionB}) the horizontal Lin
constraint eq. (\ref{lin_constraint}) can be written in terms of the canonical
Liouville 1--form $\theta_P$ on $T^*P$ as

\begin{eqnarray} \label{LC}
\left\langle B^P \left(\frac{d}{dt}(xq)\right),pq \right\rangle  &=& \left\langle
(\dot{x} - B_q(\dot{q})_P(x))q,pq \right\rangle \nonumber \\
 &=& -{\theta_P}{(x,p)} (\dot{x}, \dot{p}) - \langle J(x,p),
 B_q(\dot{q}) \rangle .
\end{eqnarray}

Finally, consider $L$ to be a $G$-invariant Lagrangian on $Q$,
this is, $L(q,\dot{q})=L(gq,g\dot{q})$ with $(q,\dot{q}) \in TQ$
and $g\in G$. As indicated above, the action of $G$ on $P$ also
permits us to define the associated bundle $T^*P\times_G Q$ over
$N$ with fiber $T^*P$. We now define the Lagrangian $\L$ on
$T(T^*P\times Q)$ by the formula:
\begin{equation}
\L (x, p, \dot{x},\dot{p}, q,\dot{q}) = L(q,\dot{q}) - \left\langle
B^P\left(\frac{d}{dt}\left(xq\right)\right),pq \right\rangle ,      \label{LM1}
\end{equation}
or using the formulas previously obtained for $B^P$ we get:
\begin{eqnarray}
\L (x, p, \dot{x},\dot{p}, q,\dot{q}) & = & L(q,\dot{q})- \left\langle
B^P\left(\frac{d}{dt}(xq)\right), pq \right\rangle \nonumber \\
\mbox{} & = & L(q,\dot{q}) + \langle J(x,p),B_q(\dot{q}) \rangle +
\theta_P(x,p)(\dot{x},\dot{p})                 \label{LM2}
\end{eqnarray}
which coincides with the Clebsh representation of the Lagrangian $L$ given in (\ref{LClebsch}).


\section{The equivalence with Lagrangian systems with symmetry}\label{main}

Now we can make precise the correspondence
between critical points of $L$ and those of $\L$.
The first result we will present is that the critical points of a Lagrangian system with symmetry $L$
can be obtained as critical points of the Clebsh Lagrangian representation of $L$ given above (\ref{LM2}).  More precisely:

\begin{theorem}\label{thClebsh}
Let $g_T\in G$, $q_0 \in Q$, $x_0\in P$, $x_T = x_0
g_T^{-1} \in P$, and $q_T=g_T q_0 \in Q$.  Then the following assertions are
equivalent:

\begin{itemize}

\item[i.-] The smooth curve $q(t)\in \Omega_{q_0,g_T G_{x_0} q_T^h}(Q)$ is a critical point of
the functional $S_L\colon \Omega_{q_0,g_T G_{x_0} q_T^h}(Q) \rightarrow
{\R}$ defined by
\begin{equation} \label{simpleL}
S_L[q] = \int_0^T L(q,\dot{q}) dt
\end{equation}

\item[ii.-] There is a smooth curve $(x(t), p(t))\in \Omega_{x_0,x_T}(T^*P)$ such that the
curve $(x(t), p(t), q(t))$ is a critical point of the functional
$S_{\L} \colon \Omega_{x_0,x_T;q_0} (T^*P \times Q) \to \R$
defined by:

\begin{eqnarray} \label{VP2} 
S_{\L} [x,p,q] &=& \int_0^T \L
(x(t),p(t), \dot{x}(t), \dot{p}(t), q(t), \dot{q}(t)) dt \nonumber \\
\mbox{} & = & \int_0^T(L(q,\dot{q}) + \langle J(x,p),B_q(\dot{q})
\rangle  + {\theta_P}{(x,p)}(\dot{x},\dot{p}))dt .
\end{eqnarray}
\end{itemize}

\end{theorem}

\begin{proof}  Because of Prop. \ref{equivalence1} there is a one-to-one correspondence $\alpha$ among the space of curves $\Omega_{q_0,g_T G_{x_0} q^h_T}(Q) $ and the space of horizontal curves $\Omega^H_{x_0,x_T;q_0}(P\times Q)$.  
Thus we may think that $L$ is defined on the space of horizontal curves $\Omega^H_{x_0,x_T;q_0}(P\times Q)$.
Now because of Thm. \ref{smooth_LMT} (Appendix B) a curve $(x(t), q(t)) \in \Omega^H_{x_0,x_T;q_0}(P\times Q)$ is a critical point of $S_L$ iff there exists a smooth lifting $(x(t),p(t), q(t))$ of this curve that is a critical point of the extended functional defined by the Lagrangian (\ref{LM1}) that due to (\ref{LM2}) is the same as the functional (\ref{VP2}).

Conversely, if  the curve $(x(t),p(t), q(t))$ is a critical point of the functional $S_\L$ defined in (\ref{VP2}), then because of
eq. (\ref{LC}), $S_\L$ is just $S_L$ plus the Lagrange multiplier determined by the submanifold of  horizontal curves $\Omega^H_{x_0,x_T;q_0}(P\times Q)$ inside the total space of curves $\Omega_{x_0,x_T;q_0} (T^*P \times Q)$. 
Thus because of the Lagrange multipliers theorem \ref{LMT} (Appendix A) we conclude that $q(t)$ is a critical point of $S_L$
in the space of curves  $\Omega_{q_0,g_T G_{x_0} q^h_T}(Q) $.
\end{proof}

And now, collecting the results obtained in Thm. \ref{thOpti} and Thm. \ref{thClebsh} above,  we can state
the following relation among the extremals of a  Lie-Scheffers-Brockett optimal control problem, the solutions of the corresponding Lagrangian system with symmetry and the critical points of the Clebsh Lagrangian associated to it.

\begin{corollary}\label{LSBC_Thm}  
Given $g_T\in G$, $q_0 \in Q$, $x_0\in P$, $q_T=g_T q_0 \in Q$, and $x_T = x_0 g_T^{-1} \in P$.
Then the following assertions are equivalent:

\begin{itemize}

\item[i.-]  (Lie-Scheffers-Brockett Optimal control system.)
The smooth curve $(x(t),q(t))$
in the subspace $\Omega_{x_0,x_T;q_0} (P\times Q)$ defined by the
equation $\dot{x} = \xi(q,\dot{q})(x)$, is a critical point of the
objective functional:
$$  S [q] = \int_0^T L(q,\dot{q}) dt .$$

\item[i.-] (Invariant Lagrangian system.)
The smooth curve $q(t)\in \Omega_{q_0,g_T G_{x_0} q^h_T}(Q)$ is a
critical point of the functional $S_L: \Omega_{q_0,g_T G_{x_0}
q^h_T}(Q) \rightarrow {\R}$ defined by
\begin{equation}
S_L[q] = \int_0^T L(q,\dot{q}) dt .
\end{equation}

\item[ii.-] (Clebsch Lagrangian system.) There is a smooth curve
$(x(t), p(t))\in \Omega_{x_0,x_T}(T^*P)$ such that the smooth curve
$(x(t), p(t), q(t))$ is a critical point of the functional $S_{\L}
\colon \Omega_{x_0,x_T;q_0} (T^*P \times Q) \to \R$ defined by:

\begin{eqnarray}
 S_{\L} [x,p,q] &=& \int_0^T \L (x(t),p(t), \dot{x}(t),
\dot{p}(t), q(t), \dot{q}(t)) dt \nonumber \\
\mbox{} & = & \int_0^T (L(q,\dot{q}) + \langle J(x,p),B_q(\dot{q})
\rangle + {\theta_P}{(x,p)}(\dot{x},\dot{p}))dt .
\end{eqnarray}

\end{itemize}
\end{corollary}

\section{Some applications and examples}\label{examples}

\subsection{Euler rigid body equations}\label{ejem1}

\subsubsection{The group $SO(3)$}
We will discuss first the simple case of Euler's equation for the
rigid body as an optimal control problem that served as a
guiding example for the discussion in \cite{Bl00}.

\vspace{0.2cm}

Let $G = SO(3)$ be the rotation group. As a principal fibre bundle
$Q$ we will consider the group $SO(3)$ acting on the left on
itself. The state space $P$ will be
the Lie group $SO(3)$ again, but this time acting on itself on the
right. The control space for the optimal control representation
problem will be the tangent bundle $TQ = TSO(3) \cong_L SO(3)
\times {\goth {so}}(3)$ where the subscript $L$ indicates that we
have used left translations on $SO(3)$ for the identification.

\vspace{0.2cm}

The Lagrangian density is the kinetic energy $L(q,\omega) = 1/2
\langle \omega, J(\omega)\rangle$, $(q,\omega) \in  SO(3) \times
{\goth {so}}(3)$ defined by a right-invariant metric $J$ on
$SO(3)$.  The connection $B$ on $Q = SO(3)$
will be simply the canonical right-invariant Maurer-Cartan 1-form
(that in the left-invariant representation of $TSO(3)$ above is the
identity matrix).
\vspace{0.2cm}

Thus, computing the vector field $B(\omega)$ on $P$ at
$x$, we get 
$$(B(\omega))_P(x) = d/dt(x \exp (-tB(\omega)))\mid_{t
= 0}=x \omega ,$$ and the control equation (\ref{controlrepres}) will
be:
$$\dot x = (B(\omega))_P(x) = x \omega.$$
Pontryagin's Hamiltonian will take the form:
$$ H(q,\omega, x, p) = \langle p, x \omega \rangle - \frac{1}{2} \langle \omega, J(\omega)\rangle ,$$
where $(q,\omega)\in TQ = TSO(3) \cong SO(3) \times \mathfrak{so}(3)$ and $(x,p)\in T^*P\cong TP \cong_R SO(3) \times \mathfrak{so}(3)$ and Pontryagin's Hamilton equations (\ref{MPP}) are then
$$ \dot x = x \omega , \quad \quad \dot p = p \omega,$$
together with the constraint condition:
$$x^Tp-p^Tx  - J \omega = 0 .$$
Finally,  Euler--Lagrange equations for the Lagrangian $L$ on $TSO(3)$ are easily obtained as:
$$ \dot{g} = g \omega, \quad J\dot{\omega} = [J \omega ,  \omega] . $$

\subsubsection{The group $SU(2)$}

Now we consider as before
$G=Q=P=SU(2)$ where we have replaced the orthogonal
group $SO(3)$ by its universal cover, the special unitary group
$SU(2)$. The Lie group
$$ SU(2) = \set{ g = \matriz{rr}{a & b \\ -\bar{b} & \bar{a}} \mid a,b \in \C, \quad a\bar{a} +b\bar{b} = 1 } $$
has Lie algebra
$$ {\goth {su}}(2) = \set{\xi_+ e_+ + \xi_- e_- + \xi_0 e_0 \mid \xi_+, \xi_-, \xi_0 \in \R } ,$$
with
$$ e_+ = \matriz{rr}{0 & 1 \\ -1 & 0}, \quad e_- = \matriz{rr}{0 & i \\ i & 0}, \quad e_0 = \matriz{rr}{i & 0
\\ 0 & -i}.$$

We shall identify $TQ$ with $SU(2) \times {\goth
{su}}(2)$ by left translations, i.e., $(g, \xi ) \in
SU(2) \times {\goth
{su}}(2)$, corresponds
to the element $(g,\dot{g}) \in TSU(2)$, with
$$ \dot{g} = g \xi.$$
The vector field $\xi_P (x)$ takes the form
 $\xi_P (x) =x \xi$.
Hence the control equation becomes
$$ \dot{x} =x \xi.$$

The objective functional is

\Eq{\label{objcontrol4} S = \frac {1}{4} \int_0^T  \langle \xi,
J(\xi) \rangle dt ,}
where we use the Killing form $\langle \xi, \zeta \rangle = 4\Tr (\xi \zeta )$.
The Euler-Lagrange equations for the Lagrangian
 \\
 \Eq{L(g, \dot{g})=\mitad \langle g^{-1}\dot{g},J( g^{-1}\dot{g})\rangle} are given again by: 
\Eq{\dot{g} = g \xi, \quad    \quad J (\dot{\xi}) = [J(\xi), \xi].} 
The equations of motion given by Pontryagin maximum principle are:
$$ \dot{x} =x \xi , \quad \dot{p} =p \xi ,$$
the relation among both sets of equation is given by:
$$x^\ast p-p^\ast x - J \xi  = 0,$$ 
and taking derivatives with
respect to $t$ we obtain 
$$ \quad  J\dot{\xi} = [J \xi ,\xi].$$


\subsection{Riccati equations}
\subsubsection{ The group $G = SL(2,\R)$ }\label{ricatti_sl2}   
We will consider now the
group $G = SL(2,\R)$ and as with the rigid body example, we will
consider the control space $Q$ the group $SL (2,\R)$ itself. To
keep the conventions held along the paper we will consider the
group $G$ acting on the left on $Q$ by left multiplication. The
Lie group
$$ SL (2,\R ) = \set{ g = \matriz{cc}{a & b \\ c & d} \mid a,b,c,d \in \R, \quad ad -bc = 1 } $$
has Lie algebra
$$ {\goth {sl}}(2, \R ) = \set{\xi_+ e_+ + \xi_- e_- + \xi_0 e_0 \mid \xi_+, \xi_-, \xi_0 \in \R } ,$$
with
$$ e_+ = \matriz{cc}{0 & 1 \\ 0 & 0}, \quad e_- = \matriz{cc}{0 & 0 \\ 1 & 0}, \quad e_0 = \matriz{cc}{1 & 0
\\ 0 & -1} ,$$
and nonzero commutation relations,
$$ [e_+, e_-] = e_0, \quad [e_+ , e_0] = - 2 e_+, \quad [e_-,e_0 ] = 2e_- .$$
We shall identify $TQ$ with $SL(2, \R ) \times {\goth {sl}}(2, \R
)$ by left translations, i.e., $(g, \xi ) \in SL(2, \R ) \times
{\goth {sl}}(2, \R )$, corresponds to the element $(g,\dot{g}) \in
TSL(2,\R )$, with
$$ \dot{g} = g \xi .$$

Now we will consider the state space $P = \R$ and $SL(2,\R )$ acting
on it by Moebius transformations, i.e.,
$$ x\cdot g^{-1} = \frac{dx - b}{-cx + a},  \quad \quad g = \matriz{cc}{a & b \\ c & d} \in SL (2, \R ) .$$

If $B$ is a diagonal $\mathfrak{sl}(2,\R)$--valued equivariant 1-form (a slight generalization of a principal connection), i.e., 
$B = \textrm{diag} (B_+,B_-, B_0)$ in the basis $e_+,e_-, e_0$, then the vector field
$\xi_P (x) = B_g (\dot{g})_P(x)$ takes the form:
$$ \xi_P(x) = (\xi_+ B_+ + 2\xi_0 B_0 x - \xi_- B_- x^2)\frac{\partial}{\partial x}.$$
Hence the control equation becomes the Riccati equation:
\Eq{\label{xcontrol} \dot{x} = a(t)x^2+b(t)x+c(t) ,}
with $a(t) = -\xi_- B_- $, $b(t) = 2\xi_0 B_0 $ and $c(t)= \xi_+ B_+$.

\vspace{0.2cm}

To define the optimal control problem we will
consider again the objective functional

\Eq{\label{objcontrol2} S = \int_0^T \mitad \langle \xi, I \xi
\rangle dt ,} where we use the Killing--Cartan form $\langle \xi, \zeta
\rangle = 4\Tr (\xi^T \zeta )$.

\vspace{0.2cm} 

The results discussed along the paper show that
there is a well defined relation among the solutions of the
optimal control problem given by eqs. (\ref{xcontrol}),
(\ref{objcontrol2}) and the critical paths of the Lagrangian
system

\Eq{\label{lagsl2} L(g, \dot{g}) = \mitad \langle g^{-1}\dot{g},I
g^{-1}\dot{g} \rangle } on $TSL(2,\R )$.

\vspace{0.2cm} 

The Euler-Lagrange equations for the Lagrangian
(\ref{lagsl2}) are given by:
\Eq{\label{ecelsl2r} \dot{g} = g \xi ,  \quad I \dot{\xi} = [I \xi, \xi] .}
The equations of motion resulting after applying Pontryagin's maximum principle to the Hamiltonian $H(x, p; g, \dot{g})= \langle p, a x^2+b x+c \rangle - \mitad \langle \xi, I \xi
\rangle$ are:
\Eq{\label{hamsl2r} \dot{x} = \dfrac{\partial H}{\partial p} = a x^2+ b x + c, \quad \dot{p} =-\dfrac{\partial H}{\partial x}= - (2a x+b) p ,}
and the optimal feedback relations are given by:
\begin{eqnarray}\label{feedback} 0 &=& \dfrac{\partial H}{\partial \xi_+}=B_+ p -I_+ \xi_+, \nonumber \\  0 &=&  \dfrac{\partial H}{\partial \xi_-}=-B_- p x^2- I_- \xi_-, \nonumber \\ 0 &=& \dfrac{\partial H}{\partial \xi_0}=2B_0 p x-I_0 \xi_0.
\end{eqnarray}

Substituting in (\ref{hamsl2r}) the values of $\xi_+, \xi_-, \xi_0$ obtained from (\ref{feedback}) we get the following nonlinear Hamiltonian equations:
$$\dot{x}=\Big[\dfrac{B_+^2}{I_+}x^4 + \dfrac{4 B_0^2}{I_0} x^2 + \dfrac{B_-^2}{I_-}\Big]~p, \quad \dot{p}=-\Big[\dfrac{2 B_+^2}{I_+} + \dfrac{4 B_0^2}{I_0}\Big]~p^2 x,$$
Thanks to the relation of reciprocity seen in this article, we can find solutions to them through Euler-Lagrange equations (\ref{ecelsl2r}) that take the form of a hyperbolic rigid body:
\Eq{\label{eulagsl2r2}\left\{\begin{array}{ll} \dot{\xi}_+=\dfrac{2(I_0-I_+)}{I_+}~ \xi_0 ~\xi_ +,\\\\
\dot{\xi}_-=-\dfrac{2(I_0-I_-)}{I_+}~ \xi_0 ~\xi_ -,\\\\
\dot{\xi}_0=\dfrac{(I_+-I_-)}{I_0}~ \xi_+ ~\xi_-.\end{array}\right.}
We can solve these equations easily in the symmetric case, i.e., $I_+=I_-=I$. Thus, $\dot{\xi}_0=0$. Let $\alpha$ be  the quantity $\alpha = 2 \xi_0 (I_0 -I)/I$.  Solving the equations (\ref{eulagsl2r2}), we obtain the analogous solution to the top precession in the hyperbolic case, which is written as: 
$$\xi_+= \xi_{+0}~e^{\alpha t}, \xi_-= \xi_{-0}~e ^{-\alpha t}.$$

Finally substituting in (\ref{feedback}) again we obtain the solutions we were looking for:
$$x(t)=\dfrac{C_0}{2 C_+}~e^{-\alpha t}, \quad p(t)=2 C_+ ~e ^{\alpha t},$$ 
where $C_0=\dfrac{I_0 \xi_{0}}{2 B_0}$ and $C_+ =\dfrac{I \xi_{+ 0}}{2 B_+}$. 

\vspace{0.2cm} \subsubsection{The group $SU(2)$ } 
If we substitute $SL(2, \R )$ by $SU(2)$ in the previous example, Section \ref{ricatti_sl2}, the
control equation becomes:
\Eq{\label{xcontrol2} \dot{x}
=a(t)x^2+b(t)x+c(t),} with $a(t)
= \xi_+ B_+- i\xi_- B_- $, $b(t) = 2 i \xi_0 B_0 $ and $c(t)= \xi_+ B_++i\xi_- B_- $.

\vspace{0.2cm}
 The Euler-Lagrange equations for the Lagrangian
(\ref{lagsl2}), using now the Killing--Cartan form of $SU(2)$, are given by:
$$ \dot{g} = g \xi ,  \quad  \dot{J}\xi = [J \xi ,\xi],$$ with
$$J(\xi)=I\xi+\xi I.$$

Repeating again the procedure above we obtain the following solutions for the symmetric case:
$$x(t)=\dfrac{C_0}{ C_-e^{-\alpha t}+i C_+ e^{\alpha t}}, \quad p(t)=C_+ ~e ^{\alpha t}-i C_ - e^{-\alpha t},$$ 
where $C_0=\dfrac{I_0 \xi_{0}}{2 B_0}$, $C_- =\dfrac{I \xi_{- 0}}{2 B_-}$, $C_+ =\dfrac{I \xi_{+ 0}}{2 B_+}$ and $\alpha$ as above.
A similar instance of this correspondence was discussed in the context of quantum optimal control in \cite{Ib08}.

\vspace{0.2cm}


\subsubsection{The group $SO(2,1)$}

If we substitute now $SL(2, \R )$ by the group $SO(2,1)$ whose Lie algebra is given by:
$${\goth {so}}(2,1) = \set{\xi_+ e_+ + \xi_- e_- + \xi_0 e_0 \mid \xi_+, \xi_-, \xi_0 \in \C },$$
with
$$ e_+ = \matriz{rr}{i & 0 \\ 0&-i}, \quad e_- = \matriz{rr}{0 & i \\-i & 0},
\quad e_0 = \matriz{rr}{0 &-1\\ -1 & 0} ,$$ the control equation
becomes:
\Eq{\label{xcontrol3} \dot{x} =a(t)x^2+b(t)x+c(t),} with $a(t)
= \xi_0 B_0+i\xi_- B_- $, $b(t) = 2 i \xi_+ B_+ $ and $c(t)= -\xi_0 B_0+i\xi_- B_-$.

\vspace{0.2cm}
 The Euler-Lagrange equations for the Lagrangian
(\ref{lagsl2}) are given by:
$$ \dot{g} = g \xi ,  \quad I \dot{\xi} = [ \xi^\ast,I\xi] . $$

Following the same procedure we obtain in the symmetric case the solutions:
$$x(t)=\dfrac{C_+ e^{\alpha t}}{ C_-e^{-\alpha t}-i C_0}, \quad p(t)=-i C_-e^{-\alpha t}- C_0,$$ 
where $C_0=\dfrac{I_0 \xi_{0}}{2 B_0}$, $C_- =\dfrac{I \xi_{- 0}}{2 B_-}$ and $C_+ =\dfrac{I \xi_{+ 0}}{2 B_+}$.


\newpage

\section*{Appendix A. A Lagrange's multiplier theorem}

In this appendix we will prove a version of
Lagrange's multiplier theorem which is suitable for
the purposes of the paper.

\begin{theorem}\label{LMT}  Let $E$ be a vector bundle with hermitian connection $\nabla$ and standard fiber the Hilbert space $\mathcal{H}$ over a smooth Hilbertian manifold $M$.
Let $s\in \Gamma(E)$ be a smooth section of
$E$ transverse to the zero section of $E$ and $W$ an open subset in the zero set of the section $s$, $Z_s = \{  x \in M \mid s(x) = 0 \}$.
Let $f\colon M \to \R$ be a $C^1$-function and $f_W\colon W \to
\R$ the restriction of $f$ to $W$.  Then they are equivalent:

\begin{itemize}

\item[i.-]  The point $x_0 \in W$ is a critical point of $f_W$.

\item[ii.-]  There exists $\alpha_0 \in E_{x_0}^*$ such that the
point $(x_0,\alpha_0) \in E^*$ is a critical point of the function $F\colon E^* \to
\R$ given by:
$$F(x,\alpha) = f(x) + \langle \alpha, s(x)\rangle .$$
\end{itemize}

\end{theorem}

\medskip

The bundle $E^*$ is the dual bundle of $E$.  The fiber $E_x^*$ at each point $x$ of $M$ is
the topological dual of the Hilbert space $E_x \cong \mathcal{H}$ and because of Riesz theorem,
it can be naturally identified with $E_x$.  In this sense $E^* \cong E$.
The function $F$  in the statement of the theorem above can be written alternatively as
$$F = \pi^* f + P_s ,$$
where $\pi\colon E\to M$ is the bundle projection and $P_s$ denoting the linear function
along the fibres of $E$ induced by the
section $s$, $P_s(x,\alpha) = \langle \alpha, s(x)\rangle$.

\bigskip

Proof.  Because $s$ is transverse to the zero section of $E$,
$Z_s$ is a smooth submanifold of $M$.  Since $W$ is
an open subset of $Z_s$, it is a smooth submanifold of $M$.
Moreover, $T_xZ_s = \ker \nabla s(x)$, $x\in Z_s$.
If $x \in W$, because $W$ is open in $s^{-1}(0)$, then we have $T_xW = T_xs^{-1}(0) = \ker \nabla s(x)$.

Let us consider now a point $x_0 \in W$ which is a critical point of $f_W$.
Then, $df_W(x_0) = 0$, i.e., $df(x_0)(U) = 0$ for all $U\in T_xW$,
hence $df(x_0) \in (\ker \nabla s(x_0))^0$.
If we compute now the differential of the function $F$ we obtain,
$$dF(x,\alpha)(X) = df(x)(\pi_*X) + d\langle \alpha, s(x)\rangle (X) = df(x)(U) + \langle X^V, s(x)\rangle +
\langle \alpha , \nabla_{U} s(x) \rangle ,$$
where the tangent vector $X \in T_{(x,\alpha)}E$ is decomposed into its horizontal and vertical components
$X = X^H + X^V$ with respect to the connection $\nabla$, and $\pi_*(X) = U \in T_xM$.
If $x_0\in W$, then $s(x_0 ) = 0$, and we get:
$$dF(x_0,\alpha)(X) = df(x_0)(U) + \langle \alpha , \nabla_U s(x) \rangle .$$
Because of the Fredholm alternative theorem the equation $\langle \alpha , \nabla s(x) \rangle = - df(x_0)$
has a solution $\alpha$ iff $df(x_0) \in (\ker \nabla s(x_0))^0$.

Conversely, if $(x_0,\alpha_0) \in E^*$ is a critical point of the function $F$
then, as $dF(x_0,\alpha_0)=0$ and $x_0\in W$, we get:
$$df(x_0) + \langle \alpha_0 , \nabla s(x_0) \rangle = 0 .$$
Again the equation $\langle \alpha , \nabla s(x) \rangle = - df(x_0)$
has a solution if and only if $df(x_0) \in (\ker \nabla s(x_0))^0$, and this implies that
$df_W(x_0) = 0$.\hfill$\Box$

\newpage

\section*{Appendix B. Lagrange's multiplier theorem on spaces of horizontal curves and optimal control}

To use Lagrange multipliers Thm. \ref{LMT} as stated in Appendix A, in the context of
this paper, requires to set up the adequate framework.    We will use for this purpose
Klingerberg's analytic setting
for functionals in spaces of curves \cite{Kl78}.

We shall consider the spaces of curves $\gamma(t) = (x(t),q(t))$, $t \in I = [0,T]$,
of Sobolev class $k$, $k\geq 1$, on the space $P\times Q$. Such space will be denoted by
$H^k([0,T],P\times Q)$ \cite{Kl78} and  is a
paracompact Hilbert manifold modelled on the Hilbert space
$H^k([0,T], \R^{n+m})$, $n = \dim\: P$, $m= \dim\: Q$, of maps $\gamma\colon
[0,T] \to \R^{n+m}$ in $L^2([0,T], \R^{n+m})$ possessing weak derivatives
$\gamma^{(l)}(t)$, $0\leq l \leq k$, in $L^2([0,T], \R^{n+m})$. The
tangent space to $H^k([0,T], P\times Q)$ at the curve $\gamma (t) =
(x(t), q(t))$ is given by the sections  of Sobolev class $k$ of
the pull-back bundle $\gamma^*(TP\times TQ)$. Because the bundle
$\gamma^*(TP\times TQ)$ over $[0,T]$ is trivial, such space of
sections can be identified with the Hilbert space of Sobolev maps
$H^k([0,T], \R^{n+m})$. We shall denote by $T(H^k([0,T], P\times Q))$ the
tangent bundle thus constructed on this space of curves.  Moreover
we can consider at each map $\gamma$ the Hilbert space of sections
of Sobolev class $l$, $0\leq l \leq k$, of the pull-back bundle
$\gamma^*(TP\times TQ)$. The collection of such spaces defines a
vector bundle over $H^k([0,T], P\times Q)$ whose fiber is given by
$H^l([0,T], \R^{n+m})$. We shall denote such vector bundle as $T^{(l)}
(H^k([0,T], P\times Q))$.

Various endpoint conditions for the curves we are considering
define Hilbert submanifolds of the Hilbert manifold $H^k([0,T], P\times Q)$.
For instance, given $x_0,x_T\in P$, the endpoint conditions considered along
the paper, $x(0) = x_0$, $x(T) = x_T$, define a Hilbert
submanifold of $H^k([0,T], P\times Q)$ that will be denoted
in what follows as $\mathcal{P}_{x_0,x_T}^k$, $k \geq 1$.   Thus:
$$\mathcal{P}_{x_0,x_T}^k = \set{(x(t),q(t)) \in H^k(P\times Q) \mid x(0) = x_0, x(T) = x_T} .$$
The tangent space to $\mathcal{P}_{x_0,x_T}^k$ is given by the Hilbert
subspace,
$$T_\gamma \mathcal{P}_{x_0,x_T}^k = \set{ (\delta x(t),\delta
q(t)) \in H^k(\gamma^*(P\times Q)) = T_\gamma H^k(P\times Q) \mid
\delta x(0) = \delta x(T) =0}.$$ 
We have defined in this way the tangent
bundle $T\mathcal{P}_{x_0,x_T}^k$.  In a similar way, we can
consider for each curve $\gamma\in \mathcal{P}^k$ the set of
sections of Sobolev class $l$, $l\leq k$, of the bundle
$\gamma^*(TP\times TQ)$ vanishing at the endpoints $x_0, x_T$. The total
space of such sections defines another vector bundle
$T^{(l)}\mathcal{P}_{x_0,x_T}^k \to \mathcal{P}_{x_0,x_T}^k$. 
Notice that we can also consider the bundle $T_{\mathcal{P}_{x_0,x_T}^k}^{(l)} (H^k([0,T], P\times Q))$
which is the restriction to $\mathcal{P}_{x_0,x_T}^k$ of the tangent bundle $T^{(l)} (H^k([0,T], P\times Q))$.
Thus for instance, the map $\gamma \mapsto \dot{\gamma}$
is a section of $T_{\mathcal{P}^k}^{(k-1)} (H^k([0,T], P\times Q))$.

\medskip

In order to apply this formalism to the optimal control problem discussed in the body of
the paper, we will consider the Hilbert manifold $M := \mathcal{P}_{x_0,x_T}^k$ of
curves of Sobolev class $k\geq 1$ on $P\times Q$ with fixed endpoints $x_0, x_T\in P$.  
We shall consider too the Hilbert manifold of curves of Sobolev
class $k\geq 1$ on the quotient space $P\times_G Q$ discussed along the paper (see Section \ref{associated}).
We have, as in the previous discussion, the tangent bundle $T^{(k-1)} H^k([0,T], P\times_G Q)$.   The natural projection
$\Pi\colon P\times Q \to P\times_G Q$ induces a projection on the corresponding spaces of curves 
that will be denoted with the same letter $\Pi \colon \mathcal{P}_{x_0,x_T}^k \to
H^k ([0,T], P\times_G Q)$. 

We shall denote by $E$ the pull-back
of the bundle $T^{(k-1)} H^k([0,T], P\times_G Q)$ to $\mathcal{P}_{x_0,x_T}^k$ along the map
$\Pi$. Notice that the fiber of $E\to \mathcal{P}_{x_0,x_T}^k$ at $\gamma
= (x,q)$ is the Hilbert space of $H^k$ sections of the bundle
$(xq)^*(T(P \times_G Q))$. Moreover, the
Hilbert bundle $E$ always admits an hermitian connection $\nabla$
\cite{La85}. Such connection can also be explicitly constructed
from a canonical global metric defined on $E$ but we will not
insist on these aspects here.

Consider now the section $s$ of the bundle $E$ defined by the map:
\begin{equation}\label{sBP}
s(\gamma)(t) = B^P \left( \frac{d}{dt} \big( x(t)q(t) \big) \right),
\end{equation}
where $\gamma(t) = (x(t), q(t))$, $B$ is a principal connection on the principal bundle $\pi_2\colon Q \to N$, and $B^P$ the connection associated to $B$ on the bundle $\pi_Q\colon P\times_G Q \to N$ (see Section \ref{associated_connections}). Clearly the section $s$ is smooth and transverse to the zero section of $E$.   Notice that
the tangent space at a zero section point $\gamma$ of $E$ can be
written as $T_\gamma E = T_\gamma^{(k-1)}\mathcal{P}_{x_0,x_T}^k \oplus
V_\gamma(E)$, where $V_\gamma(E) \cong E_\gamma$ denotes the vertical subspace of $T_\gamma E$.   But any vertical vector $\xi \in E_\gamma$ is in the range of $\nabla s$ for a generic connection $\nabla$.


Finally we will consider the $C^1$-map $S\colon M \to \mathbb{R}$ defined by eq. (\ref{objective}).
Hence, applying Lagrange's multiplier theorem, Thm. \ref{LMT}, to the function $S$ and the section of $E\to M$ 
defined by eq. (\ref{sBP}), we have that
the curve $\gamma$ will be a critical point of the function
$S$ restricted to the submanifold defined
by the zero set $Z_s$ of the section $s$ if and only if there exists an element
$p\in E^*$ such that $(\gamma,p)$ is a critical point of the
function $F\colon E^*\to \R$ given by
$$F(\gamma, p) = S(\gamma) +\langle p, s(\gamma )\rangle .$$
Because of the Sobolev embedding
theorem the critical curve $\gamma(t)$ is of differentiability
class $C^k$ and $p(t)$ is of differentiability class $C^r$ with $r = k - 1$.  Thus if we assume that $\gamma$ is in $\mathcal{P}_{x_0,x_T}^k$ for all
$k\geq 0$, then the critical pair $(\gamma, p)$ will be of class $(k,k-1)$ for all
$k\geq 0$, hence of class $C^\infty$.

Thus we can summarize the previous discussion in the form of the following theorem.

\begin{theorem}\label{smooth_LMT} With the notation above, a smooth curve
$\gamma(t)=(x(t), q(t))$ is a critical point of the functional
$S = \int_0^T L(x,q) dt$ subjected to the horizontal constraint conditions
$$ B^P\left( \frac{d}{dt}\big( x(t)q(t)\big) \right) = 0 $$
if and only if there exists a smooth lifting $(\gamma(t), p(t))$
of this curve that is a critical point of the extended
functional
$$S_{\L} [x,p,q]= \int_0^T \left(L(x,q) + \Big\langle p(t), B^P\Big(\frac{d}{dt}\big(x(t)q(t) \big) \Big) \Big\rangle \right) dt.$$

\end{theorem}

\end{document}